\documentclass[a4paper,twoside,10pt]{amsart}
\usepackage{amsmath,amsfonts,amssymb,amsthm,mathrsfs}
\usepackage{color,graphicx}
\usepackage[a4paper]{geometry}
\usepackage[colorlinks=true,pdfstartview={XYZ null null 1.00},pdfview=FitH,citecolor=blue,urlcolor=customgreen]{hyperref}
\newtheorem{theorem}{Theorem}[section]
\newtheorem{lemma}{Lemma}[section]
\newtheorem{conjecture}{Conjecture}[section]
\numberwithin{equation}{section}
\def\im{\mathrm{i}}
\def\d{\mathrm{d}}
\def\e{\mathrm{e}}
\def\O{\mathcal{O}}

\def\Ai{\operatorname{Ai}}
\def\Bi{\operatorname{Bi}}
\definecolor{customgreen}{rgb}{0.0, 0.5, 0.0}
\author[G. Nemes]{Gerg\H{o} Nemes}
\email{nemes.gergo@renyi.hu}
\address{Alfr\'ed R\'enyi Institute of Mathematics, Re\'altanoda utca 13--15, Budapest H-1053, Hungary}


\begin{document}

\title[On the real zeros of the cylinder and Airy functions]{Proofs of two conjectures on the real zeros\\ of the cylinder and Airy functions}

\begin{abstract} We prove the enveloping property of the known divergent asymptotic expansions of the large real zeros of the cylinder and Airy functions, and thereby answering in the affirmative two conjectures posed by Elbert and Laforgia and by Fabijonas and Olver, respectively. The essence of the proofs is the construction of analytic functions that return the zeros when evaluated along certain discrete sets of real numbers. By manipulating contour integrals of these functions, we derive the asymptotic expansions of the large zeros truncated after a finite number of terms plus remainders that can be estimated efficiently. The conjectures are then deduced as corollaries of these estimates. An analogous result for the associated phase function is also discussed.
\end{abstract}

\maketitle

\section{Introduction}

Let $\nu$ and $\alpha$ be real numbers, $0\leq \alpha<1$. We define the (general) cylinder function $\mathscr{C}_\nu  (z,\alpha )$ of order $\nu$ by
\[
\mathscr{C}_\nu  (z,\alpha ) = J_\nu  (z)\cos (\pi \alpha ) + Y_\nu  (z)\sin (\pi \alpha ).
\]
Here $J_\nu  (z)$ and $Y_\nu  (z)$ denote the Bessel functions of the first and second kind, respectively (see, for instance, \cite[\href{http://dlmf.nist.gov/10.2.ii}{\S10.2(ii)}]{DLMF}). In general, $\mathscr{C}_\nu  (z,\alpha )$ is a multivalued function of $z$. The principal branch corresponds to the principal branches of $J_{\pm \nu}(z)$, with a cut in the $z$-plane along the interval $(-\infty,0]$. It is well known \cite[Ch. XV, \S15.24]{Watson1944} that the cylinder function has an infinite number of positive real zeros, all of which are simple. We denote the $k$th positive zero of $\mathscr{C}_\nu  (z,\alpha )$, arranged in ascending order, by $j_{\nu ,\kappa }$ with $\kappa  = k + \alpha  > \frac{1}{2}(\left| \nu  \right| - \nu )$ and $k$ a non-negative integer (see Section \ref{Section2} for more details). A classical result of McMahon \cite{McMahon1894}\cite[Ch. XV, \S15.53]{Watson1944} states that for fixed $\nu$, the large zeros of the function $\mathscr{C}_\nu  (z,\alpha )$ are given by the asymptotic expansion
\begin{equation}\label{McMahonseries}
j_{\nu ,\kappa } \sim \beta _{\nu ,\kappa }  + \sum\limits_{n = 1}^\infty \frac{c_n (\nu )}{\beta _{\nu ,\kappa }^{2n - 1}}  = \beta _{\nu ,\kappa }  - \frac{4\nu ^2  - 1}{8\beta _{\nu ,\kappa } } - \frac{(4\nu ^2  - 1)(28\nu ^2  - 31)}{384\beta _{\nu ,\kappa }^3} +  \cdots 
\end{equation}
as $k\to +\infty$, where $\beta _{\nu ,\kappa }  = \left( \kappa  + \frac{1}{2}\nu  - \frac{1}{4} \right)\pi$. In Appendix \ref{Appendix}, we show that the coefficients $c_n (\nu )$ are polynomials in $\nu^2$ of degree $n$ and provide a recurrent scheme for their evaluation. The subject of our interest is the behaviour of the error terms associated with the expansion \eqref{McMahonseries} when $\nu$ is confined to the interval $-\frac{1}{2} < \nu < \frac{1}{2}$. Note that
\[
\mathscr{C}_{ \pm 1/2} (z,\alpha ) = \sqrt {\frac{2}{\pi z}} \cos \left( z - \left( \alpha  \pm \tfrac{1}{4} + \tfrac{1}{4} \right)\pi \right)
\]
(cf. \cite[\href{http://dlmf.nist.gov/10.16.E1}{Eq. 10.16.1}]{DLMF}), and therefore $j_{\pm 1/2 ,\kappa } = \beta _{\pm 1/2 ,\kappa }$.

The first result in this direction is due to Watson \cite[Ch. XV, \S15.33]{Watson1944} (extending an earlier result of Schafheitlin), who showed that the positive zeros of $J_\nu (z)$ (i.e., when $\alpha=0$) satisfy
\begin{equation}\label{Watsonbound}
\beta _{\nu ,k}  < j_{\nu ,k} 
\end{equation}
for all positive integer values of $k$ and $-\frac{1}{2} < \nu < \frac{1}{2}$. Improving on Watson's estimate, F\"orster and Petras \cite{Forster1993} obtained the following remarkable inequalities:
\begin{equation}\label{bounds}
\beta _{\nu ,k}  - \frac{4\nu ^2  - 1}{8\beta _{\nu ,k} } - \frac{(4\nu ^2  - 1)(28\nu ^2  - 31)}{384\beta _{\nu ,k}^3 } < j_{\nu ,k}  < \beta _{\nu ,k }  - \frac{4\nu ^2  - 1}{8\beta _{\nu ,k } },
\end{equation}
for any positive integer $k$ and $-\frac{1}{2} < \nu < \frac{1}{2}$ (compare \eqref{McMahonseries}). We remark that the upper bound in \eqref{bounds}, under the more restrictive assumption $0 \leq \nu < \frac{1}{2}$, was also proved earlier by Hethcote \cite{Hethcote1970}.

Motivated by the inequalities \eqref{Watsonbound} and \eqref{bounds}, Elbert and Laforgia \cite{Elbert2001} formulated the following conjecture.

\begin{conjecture}\label{Elbert}
For $-\frac{1}{2} < \nu < \frac{1}{2}$, an even (respectively odd) number of terms of McMahon's expansion always gives upper (respectively lower) bounds for $j_{\nu ,\kappa }$.
\end{conjecture}

Let us emphasise that the inequalities \eqref{Watsonbound} and \eqref{bounds} were proved in the special case that $\alpha=0$, whereas the conjecture was proposed for the general situation when $0\leq \alpha<1$. It can be verified by direct numerical computation that the conjecture does not necessarily hold if $\beta _{\nu ,\kappa }<0$ (which can only happen if $k=0$). On the other hand, graphical depiction suggests that $( - 1)^n c_n (\nu )<0$ holds for all positive integer $n$ and $-\frac{1}{2} < \nu < \frac{1}{2}$ (cf. Figure \ref{fig4}), whence the asymptotic expansion \eqref{McMahonseries} is likely of alternating type when $\beta _{\nu ,\kappa }>0$, giving some evidence for the conjecture being true for positive values of $\beta _{\nu ,\kappa }$.

One of our main goals is to establish Conjecture \ref{Elbert} under the assumption that $\beta _{\nu ,\kappa }>0$. In  particular, we shall prove the following theorem.

\begin{theorem}\label{Elberttheorem} Assume that $-\frac{1}{2} < \nu < \frac{1}{2}$ and $\beta _{\nu ,\kappa }>0$. Then for all positive integers $n$ and $N$, $( - 1)^n c_n (\nu )<0$ and the $N$th error term of McMahon's expansion \eqref{McMahonseries} (that is, the error on stopping the expansion \eqref{McMahonseries} at $n = N-1$) does not exceed the first neglected term in absolute value and has the same sign as that term.
\end{theorem}

Throughout this paper, if not stated otherwise, empty sums are taken to be zero. In the language of enveloping series \cite[Ch. 4, \S1]{Polya1998}, Theorem \ref{Elberttheorem} states that the expansion \eqref{McMahonseries} envelopes the zeros $j_{\nu ,\kappa }$ in the strict sense. Theorem \ref{Elberttheorem} will be established as a direct consequence of the more general Theorem \ref{thmX} given in Section \ref{Section2}. Our approach is completely different from that of F\"orster and Petras since their method does not seem to extend to the general situation.

It is natural to ask if a statement similar to Theorem \ref{Elberttheorem} holds for other values of $\nu$. We briefly discuss this problem in Section \ref{Section6}. For a survey on further properties of the zeros of cylinder functions, the interested reader is referred to \cite{Elbert20011}.

The second family of functions we are interested in consists of certain linear combinations of Airy functions. More precisely, we consider, for $0\leq \alpha<1$ and any complex $z$, the entire function
\begin{equation}\label{genAiry}
\mathscr{A}(z,\alpha ) = \Ai(z)\cos (\pi \alpha ) + \Bi(z)\sin (\pi \alpha ),
\end{equation}
where $\Ai(z)$ and $\Bi(z)$ are the Airy functions of the first and second kind, respectively (see, e.g., \cite[\href{http://dlmf.nist.gov/9.2}{\S9.2}]{DLMF}). It is known \cite{Gil2014} that $\mathscr{A}(z,\alpha )$ has an infinite number of negative zeros. We denote them by $a_\kappa$, arranged in ascending order of absolute value with $\kappa  = k - \alpha  > \frac{1}{6}$ and $k$ a positive integer (hence indexing starts at $k = 1$ or $k = 2$ according to whether $\alpha  < \frac{5}{6}$ or $\alpha  \ge \frac{5}{6}$, see Section \ref{Section2} for more details). When $\alpha=\frac{1}{2}$, $a_\kappa$ is precisely the $k$th negative zero of the Airy function $\Bi(z)$ and we adopt the standard notation $b_k$ for this particular case. The large negative zeros of the function \eqref{genAiry} are known to posses the divergent asymptotic expansion
\begin{equation}\label{airyzerosasymp}
a_\kappa   \sim  - \gamma _\kappa ^{2/3} \left( 1 + \sum\limits_{n = 1}^\infty  \frac{T_n}{\gamma _\kappa ^{2n}} \right) =  - \gamma _\kappa ^{2/3} \left( 1 + \frac{5}{48\gamma _\kappa ^2 } - \frac{5}{36\gamma _\kappa ^4 } + \frac{77125}{82944\gamma _\kappa ^6 } -  \frac{108056875}{6967296\gamma _\kappa ^8 }+\cdots \right)
\end{equation}
as $k\to +\infty$, where $\gamma _\kappa   = \frac{3}{8}\pi (4\kappa  - 1)$. The coefficients $T_n$ are rational numbers and can be computed
from recursive formulae (see Appendix \ref{Appendix}). In the special cases of $a_k$ ($\alpha=0$) and $b_k$ ($\alpha=\frac{1}{2}$), this is a classical result and is presumably due to Miller \cite{Miller1946}. For general $0\leq \alpha<1$, \eqref{airyzerosasymp} was established more recently by Gil and Segura \cite{Gil2014}. Our point of interest is the behaviour of the remainder terms associated with the expansion \eqref{airyzerosasymp}.

The most important result in this direction is due to Pittaluga and Sacripante \cite{Pittaluga1991}. For the particular cases of $a_k$ and $b_k$, they showed that the $N$th error term (that is, the error on stopping the expansion \eqref{airyzerosasymp} at $n = N-1$) does not exceed the first neglected term in absolute value and has the same sign as this term when $N = 1, 2, 3, 4, 5$, and also that the sixth error term has the opposite sign to the fifth term.

With the help of computer algebra system (Maple V), Fabijonas and Olver \cite{Fabijonas1999} verified the inequality $(-1)^n T_n<0$ up to $n=99$, whence the asymptotic expansion \eqref{airyzerosasymp} is likely of alternating type. This observation and the results by Pittaluga and Sacripante led them to make the following conjecture.

\begin{conjecture}\label{Olver}
In the expansions of $a_k$ and $b_k$, the $N$th error term is bounded by the first neglected term and has the same sign for all values of $N\geq 1$.
\end{conjecture}

The second main aim of this paper is to prove Conjecture \ref{Olver} for the case $\gamma_\kappa>0$ and all $0\leq \alpha<1$. In  particular, we will prove the following theorem.

\begin{theorem}\label{Olvertheorem} Assume that $\gamma_{\kappa }>0$. Then for all positive integers $n$ and $N$, $( - 1)^n T_n <0$ and the $N$th error term of the expansion \eqref{airyzerosasymp} does not exceed the first neglected term in absolute value and has the same sign as that term.
\end{theorem}

In other words, Theorem \ref{Olvertheorem} asserts that the expansion \eqref{airyzerosasymp} envelopes the zeros $a_{\kappa }$ in the strict sense. Theorem \ref{Olvertheorem} is a direct consequence of the more general Theorem \ref{thmT} stated in Section \ref{Section2}. Our method of proof is rather different from that of Pittaluga and Sacripante which is based on a Sturm-type theorem, due to Hethcote \cite{Hethcote19702}. They verified their result for each $N$ separately, and as $N$ increases, their procedure becomes cumbersome and increasingly laborious.

The remaining part of the paper is structured as follows. In Section \ref{Section2}, we state our main results. The proofs are provided in Sections \ref{Section3}--\ref{Section5}. The paper concludes with a discussion in Section \ref{Section6}.

\section{Main results}\label{Section2}

To state our results, we introduce some further notation. If $H_\nu ^{(1)} (z)$ denotes the Hankel function of the first kind \cite[\href{http://dlmf.nist.gov/10.2.ii}{\S10.2(ii)}]{DLMF}, then for any $\nu \geq 0$ and $z>0$, we write
\begin{equation}\label{thetaMdef}
M_\nu  (z)\e^{\im\theta _\nu  (z)}  = H_\nu ^{(1)} (z)
\end{equation}
where the modulus function $M_\nu (z)$ ($>0$) of order $\nu$ and the phase function $\theta_\nu (z)$ of order $\nu$ are continuous real functions of $\nu$ and $z$, with the branch of $\theta _\nu  (z)$ fixed by
\[
\mathop {\lim }\limits_{z \to 0 + } \theta _\nu  (z) =  - \frac{\pi }{2}.
\]
For basic properties of these functions, we refer the reader to \cite[\href{http://dlmf.nist.gov/10.18}{\S10.18}]{DLMF}. Both $M_\nu (z)$ and $\theta_\nu (z)$ can be extended to negative values of their order $\nu$ by appealing to the connection formula $H_{ - \nu }^{(1)} (z) = \e^{\pi \nu \im} H_\nu ^{(1)} (z)$ \cite[\href{http://dlmf.nist.gov/10.4.E6}{Eq. 10.4.6}]{DLMF}. Hence,
\begin{equation}\label{phaserelation}
M_{ - \nu } (z) = M_\nu  (z) \quad \text{and} \quad \theta _{ - \nu } (z) = \theta _\nu  (z) + \pi \nu
\end{equation}
for all positive $\nu$ and $z$. With these definitions,
\begin{equation}\label{Jviaphase}
J_\nu  (z) = M_\nu  (z)\cos \theta _\nu  (z) \quad \text{and} \quad Y_\nu  (z) = M_\nu  (z)\sin \theta _\nu (z)
\end{equation}
\cite[\href{http://dlmf.nist.gov/10.18.E4}{Eq. 10.18.4}]{DLMF}, and therefore  
\begin{equation}\label{Cphasemodulus}
\mathscr{C}_\nu  (z,\alpha ) = M_\nu  (z)\cos (\theta _\nu  (z) - \pi \alpha )
\end{equation}
for any real $\nu$ and $z>0$. Now let $z^\ast$ be a positive root of the cylinder function $\mathscr{C}_\nu  (z,\alpha )$. Since $M_\nu(z)$ is positive, we can infer that 
\[
\theta _\nu  (z^\ast) = \left(k +\alpha- \tfrac{1}{2}\right)\pi
\]
for some integer $k$. Because the range of $\theta _\nu (z)$ consists of real numbers that are greater than $
 - \frac{\pi }{2}(\nu  - \left| \nu  \right| + 1)$, the inequality
\begin{equation}\label{krequirement}
k  + \alpha  >  \tfrac{1}{2}(\left| \nu  \right| -\nu)
\end{equation}
must hold. From
\begin{equation}\label{Mtheta}
\theta '_\nu  (z) = \frac{2}{\pi }\frac{1}{z M_\nu ^2 (z)} > 0
\end{equation}
(see, e.g., \cite[\href{http://dlmf.nist.gov/10.18.E8}{Eq. 10.18.8}]{DLMF}), it follows that the phase function $\theta_\nu  (z)$ is a strictly monotonically increasing function of $z$ on the positive real axis. Consequently, for any integer $k$ fulfilling \eqref{krequirement}, there is a unique pre-image of $\left(k +\alpha- \tfrac{1}{2}\right)\pi$ with respect to $\theta_\nu  (z)$ and this pre-image is a positive root of $\mathscr{C}_\nu  (z,\alpha )$. Furthermore, all positive roots can be obtained in this way. Taking into account the monotonicity of $\theta_\nu  (z)$, we may then claim that
\[
\theta _\nu  (j_{\nu ,\kappa } ) = \left( k + \alpha  - \tfrac{1}{2} \right)\pi  = \left( \kappa  - \tfrac{1}{2} \right)\pi  = \beta _{\nu ,\kappa }  - \left( \tfrac{1}{2}\nu  + \tfrac{1}{4} \right)\pi .
\]
Since $\theta _\nu  (z) + \left( \frac{1}{2}\nu  + \frac{1}{4} \right)\pi$ is strictly monotonically increasing with respect to $z$, it has an inverse function $X_\nu(w)$, say. For real $\nu$ and $w > \left( \frac{1}{2}|\nu|  - \frac{1}{4} \right)\pi$, $X_\nu(w)$ is a continuous, positive real-valued function of $\nu$ and $w$ that satisfies
\[
X_\nu (\beta _{\nu ,\kappa }) = j_{\nu ,\kappa }.
\]
Thus, McMahon's formula \eqref{McMahonseries} yields an asymptotic series for $X_\nu (\beta _{\nu ,\kappa })$ when $k$ (or equivalently $\beta _{\nu ,\kappa }$) becomes large. To establish Theorem \ref{Elberttheorem}, we shall derive the large-$w$ asymptotic expansion of $X_\nu (w)$ truncated after a finite number of terms plus a remainder that can be estimated efficiently. The derivation involves complex integration and it requires the analytic continuation of $X_\nu (w)$ to complex values of $w$. For the latter purpose, we first need to extend the phase function $\theta_\nu(z)$ to complex $z$. Our technique can also be applied to obtain computable error bounds for the know large-$z$ asymptotic series of $\theta_\nu(z)$ \cite[\href{http://dlmf.nist.gov/10.18.E18}{Eq. 10.18.18}]{DLMF}, provided $-\frac{3}{2} \leq \nu \leq \frac{3}{2}$, without much further effort. Since to our best knowledge, no such estimates have been given in the literature prior to this paper, it is worthwhile to include the proof and to present the result in a separate theorem, as follows.

\begin{theorem}\label{thmtheta} If $-\frac{3}{2} \leq \nu \leq \frac{3}{2}$, the phase function $\theta_\nu(z)$ extends analytically to the right half-plane $\Re z > 0$. For any positive integer $N$, this extended function admits the expansion
\begin{equation}\label{thetaexpexact}
\theta _\nu  (z) = z - \left( \tfrac{1}{2}\nu  + \tfrac{1}{4} \right)\pi  + \sum\limits_{n = 1}^{N - 1} \frac{t_n (\nu )}{z^{2n - 1} }  + R_N^{(\theta )} (\nu ,z)
\end{equation}
where the remainder term $R_N^{(\theta )} (\nu ,z)$ satisfies
\[
\left| R_N^{(\theta )} (\nu ,z) \right| \le \frac{\left| t_N (\nu ) \right|}{\left| z \right|^{2N - 1} } \times \begin{cases} 1 & \text{ if } \; \left|\arg z\right| \leq \frac{\pi}{4}, \\ |\csc (2\arg z)| & \text{ if } \; \frac{\pi}{4} < \left|\arg z\right| < \frac{\pi}{2}. \end{cases}
\]
In addition, if $z>0$, the remainder term does not exceed the first neglected term in absolute value and has the same sign. The coefficients $t_n(\nu)$ are polynomials in $\nu^2$ of degree $n$ and satisfy
\[
( - 1)^n t_n (\nu ) \begin{cases} >0 & \text{ if } \; -\frac{1}{2} < \nu < \frac{1}{2}, \\ =0 & \text{ if } \; \nu =\pm \frac{1}{2}, \\ <0 & \text{ if } \; \frac{1}{2} < \pm\nu \leq \frac{3}{2}. \end{cases}
\]
\end{theorem}

We are now in position to state the first main result of this paper.

\begin{theorem}\label{thmX} If $-\frac{1}{2} < \nu < \frac{1}{2}$, the function $X_\nu(w)$ extends analytically to the right half-plane $\Re w > 0$. For any positive integer $N$, this extended function admits the expansion
\begin{equation}\label{Xtruncexp}
X_\nu  (w) = w + \sum\limits_{n = 1}^{N - 1} \frac{c_n (\nu )}{w^{2n - 1} }  + R_N^{(X)} (\nu ,w)
\end{equation}
where the remainder term $R_N^{(X)} (\nu, w)$ can be bounded as follows:
\[
\left| R_N^{(X)} (\nu ,w) \right| \le \frac{\left| c_N (\nu ) \right|}{\left| w \right|^{2N - 1} } \times \begin{cases} 1 & \text{ if } \; \left|\arg w\right| \leq \frac{\pi}{4}, \\ |\csc (2\arg w)| & \text{ if } \; \frac{\pi}{4} < \left|\arg w\right| < \frac{\pi}{2}. \end{cases}
\]
In addition, if $w>0$, the remainder term does not exceed the first neglected term in absolute value and has the same sign. The coefficients $c_n(\nu)$ are polynomials in $\nu^2$ of degree $n$ with the property that $( - 1)^n c_n (\nu ) < 0$ for all $-\frac{1}{2} < \nu < \frac{1}{2}$.
\end{theorem}

Theorem \ref{Elberttheorem} follows as a corollary of this theorem by choosing $w = \beta _{\nu ,\kappa } >0$.

Let us now turn our attention to the zeros of the function $\mathscr{A}(z,\alpha )$. It is shown in Appendix \ref{Appendix2} that $\mathscr{A}( - z,\alpha )$ is expressible in terms of the modulus and phase functions of order $\frac{1}{3}$, namely
\begin{equation}\label{Aviaphase}
\mathscr{A}( - z,\alpha ) = \sqrt {\frac{z}{3}} M_{1/3} \left(\tfrac{2}{3}z^{3/2}\right)\cos \left( \theta _{1/3} \left(\tfrac{2}{3}z^{3/2}\right) + \pi \left( \alpha  + \tfrac{1}{6}\right)\right).
\end{equation}
Starting with this formula, an argument similar to that in the case of the cylinder function shows that $a_\kappa   =  - T(\gamma _\kappa  )$, where
\begin{equation}\label{Tdef}
T(w) = \left( \tfrac{3}{2}X_{1/3} \left(\tfrac{2}{3}w\right) \right)^{2/3} 
\end{equation}
and $\kappa = k-\alpha>\frac{1}{6}$, $k$ being a positive integer. In \eqref{Aviaphase} and \eqref{Tdef} the fractional powers are taking their principal values. Theorem \ref{Olvertheorem} is a special case of the following more general statement, which is our second main result.

\begin{theorem}\label{thmT} The function $T(w)$ extends analytically to the right half-plane $\Re w > 0$. For any positive integer $N$, this extended function admits the expansion
\begin{equation}\label{Texpexact}
T(w) = w^{2/3}\left( 1+ \sum\limits_{n = 1}^{N - 1} \frac{T_n}{w^{2n} }  + R_N^{(T)} (w)\right)
\end{equation}
where the remainder term $R_N^{(T)} (w)$ satisfies
\[
\left| R_N^{(T)} (w) \right| \le \frac{\left| T_N \right|}{\left| w \right|^{2N} } \times \begin{cases} 1 & \text{ if } \; \left|\arg w\right| \leq \frac{\pi}{4}, \\ |\csc (2\arg w)| & \text{ if } \; \frac{\pi}{4} < \left|\arg w\right| < \frac{\pi}{2}. \end{cases}
\]
In addition, if $w>0$, the remainder term does not exceed the first neglected term in absolute value and has the same sign. The coefficients $T_n$ are rational numbers and satisfy $( - 1)^n T_n < 0$ for all positive integer $n$.
\end{theorem}

\paragraph{\textbf{Remark.}} Following the notation in \cite{Fabijonas2004} and \cite[\href{http://dlmf.nist.gov/9.8}{\S9.8}]{DLMF}, a representation for the function $\mathscr{A}(z,\alpha )$  analogous to \eqref{Cphasemodulus} is as follows:
\[
\mathscr{A}(z,\alpha ) = M(z)\sin (\theta (z) + \pi \alpha ),
\]
where, by \eqref{Aviaphase}, 
\[
\theta (z) = \theta _{1/3} \left( \tfrac{2}{3}( - z)^{3/2} \right) + \tfrac{2\pi}{3} \quad \text{and} \quad M(z) = \sqrt { - \frac{z}{3}} M_{1/3} \left( \tfrac{2}{3}( - z)^{3/2} \right)
\]
for any $z<0$. Similarly, the derivative of $\mathscr{A}(z,\alpha )$ may be expressed in the form
\[
\mathscr{A}'(z,\alpha ) = N(z)\sin (\phi (z) + \pi \alpha )
\]
(compare \cite[\href{http://dlmf.nist.gov/9.8.E5}{Eq. 9.8.5} and \href{http://dlmf.nist.gov/9.8.E6}{Eq. 9.8.6}]{DLMF}), where
\[
\phi (z) = \theta _{2/3} \left( \tfrac{2}{3}( - z)^{3/2} \right) + \tfrac{\pi}{3} \quad \text{and} \quad N(z) =  - \frac{z}{\sqrt 3}M_{2/3} \left( \tfrac{2}{3}( - z)^{3/2} \right),
\]
and $z$ is any negative real number. The proof of this representation relies on the known relation between the derivatives of the Airy functions and the Bessel functions of order $\pm \frac{2}{3}$ (see, \cite[\href{http://dlmf.nist.gov/9.6.E7}{Eq. 9.6.7} and \href{http://dlmf.nist.gov/9.6.E9}{Eq. 9.6.9}]{DLMF}), and is along the lines of the proof of \eqref{Aviaphase}.

\smallskip

\paragraph{\textbf{Remark.}} For asymptotic expansions, there are two main ways of deriving error bounds. One set of methods starts with an integral representation (see, e.g., \cite{Bennett2018,Wong1980}), the other uses linear differential equations \cite{Olver19972}. The difficulty of our problems is that the implicitly defined functions $\theta_\nu(z)$, $X_\nu(w)$ and $T(w)$ neither have integral representations to start with nor do they satisfy any linear differential equation.

\section{Proof of Theorem \ref{thmtheta}}\label{Section3}

In this section, we prove Theorem \ref{thmtheta}. We begin by stating and proving two lemmata.

\begin{lemma}\label{lemma1} When $-\frac{3}{2} \leq \nu \leq \frac{3}{2}$, the phase function $\theta_\nu(z)$ has a holomorphic extension to the half-plane $\Re z> 0$. Moreover, if $-\frac{3}{2} < \nu < \frac{3}{2}$, this function extends continuously to the boundary $\Re z = 0$.
\end{lemma}

\begin{proof} Suppose that $z>0$. Using the definition of $\theta _\nu  (z)$ and the known relation between the Hankel functions and the modified Bessel function of the second kind \cite[\href{http://dlmf.nist.gov/10.27.E8}{Eq. 10.27.8}]{DLMF}, we derive the representation\vspace{-3pt}
\[
\e^{2\im\theta _\nu  (z)}  = \frac{M_\nu  (z)\e^{\im\theta _\nu  (z)} }{M_\nu  (z)\e^{ - \im\theta _\nu  (z)} } = \frac{H_\nu ^{(1)} (z)}{\overline{H_\nu ^{(1)} (z)}} = \frac{H_\nu ^{(1)} (z)}{H_\nu ^{(2)} (z)} =  \e^{ - \pi \im(\nu+1) } \frac{K_\nu  \left( z\e^{ - \frac{\pi }{2}\im}  \right)}{K_\nu  \left( z\e^{\frac{\pi }{2}\im} \right)}.
\]
The choice $\arg (-1)=-\pi$ has been made so as to match the limiting value of $\theta _\nu  (z)$ at the origin (see \eqref{thetalimit} below). Since $K_\nu(z)$ has no zeros when $-\frac{3}{2} \leq \nu \leq \frac{3}{2}$ and $|\arg z|< \pi$ \cite[Ch. XV, \S15.7]{Watson1944}, the rightmost fraction has a holomorphic logarithm in the (simply connected) domain $\Re z>0$. Accordingly,
\begin{equation}\label{thetaext}
\theta _\nu  (z) = \frac{1}{2\im}\log \left( \e^{ - \pi \im(\nu+1) } \frac{K_\nu  \left( z\e^{ - \frac{\pi }{2}\im}  \right)}{K_\nu  \left( z\e^{\frac{\pi }{2}\im}\right)} \right)
\end{equation}
yields the required holomorphic extension of the phase function $\theta_\nu(z)$ to the right half-plane $\Re z>0$. Here, $\log$ stands for the principal branch of the logarithm. Again, the branch of the logarithm has been chosen so as to match the limiting value of $\theta _\nu  (z)$ at $z=0$ (cf. \eqref{thetalimit} below). Furthermore, if $\nu \neq \pm\frac{3}{2}$, $K_\nu(z)$ does not vanish on the rays $\arg z =\pm \pi$ either. Thus, formula \eqref{thetaext} also defines a continuous function of $z$ when $-\frac{3}{2} < \nu < \frac{3}{2}$, $\Re z\geq 0$ and $z\neq 0$. To verify the continuity at the origin, one may use the known behaviour of $K_\nu(z)$ near $z=0$ \cite[\href{http://dlmf.nist.gov/10.30.i}{\S10.30(i)}]{DLMF} to derive
\begin{equation}\label{thetalimit}
\mathop {\lim }\limits_{z \to 0} \frac{1}{2\im}\log \left( \e^{ - \pi \im(\nu+1) } \frac{K_\nu  \left( z\e^{ - \frac{\pi }{2}\im}  \right)}{K_\nu  \left( z\e^{\frac{\pi }{2}\im}\right)} \right) = \frac{1}{2\im}\log (\e^{ - \pi \im(\nu - \left| \nu  \right| + 1)}  ) = - \frac{\pi }{2}(\nu  - \left| \nu  \right| + 1),
\end{equation}
in accordance with the definition of $\theta_\nu (z)$.
\end{proof}

\paragraph{\textbf{Remark.}} Since for $z>0$,
\begin{align*}
H_{ \pm 3/2}^{(1)} (z) & = \sqrt {\frac{2}{\pi z}} \exp \left( \im z - \left(  \pm \tfrac{3}{4} + \tfrac{1}{4} \right)\pi \im \right)\left( 1 + \frac{\im}{z} \right) \\ & = \sqrt {\frac{2}{\pi z}\left( 1 + \frac{1}{z^2} \right)} \exp \left( \im\left( z - \left(  \pm \tfrac{3}{4} + \tfrac{1}{4} \right)\pi  + \arctan \left( \frac{1}{z} \right) \right) \right),
\end{align*}
it follows that
\begin{equation}\label{32case}
\theta _{ \pm 3/2} (z) = z - \left(  \pm \tfrac{3}{4} + \tfrac{1}{4} \right)\pi  + \arctan \left( \frac{1}{z} \right).
\end{equation}
This formula gives the analytic extensions of the phase functions $\theta _{ \pm 3/2} (z) $ to the half-plane $\Re z> 0$ and indicates that they possess branch point singularities located at $z=\im$ and $z=-\im$ (which are zeros of $H_{ \pm 3/2}^{(2)} (z)$ and $H_{ \pm 3/2}^{(1)} (z)$, respectively).

At this point it proves convenient to introduce the function $\Theta _\nu  (z)$ via
\begin{equation}\label{Thetadef}
\Theta _\nu  (z) = \theta _\nu  (z) + \left( \tfrac{1}{2}\nu  + \tfrac{1}{4} \right)\pi .
\end{equation}
Clearly, this function inherits the continuity and analytic properties of $\theta _\nu  (z)$ stated in Lemma \ref{lemma1}. In particular,
\begin{equation}\label{Thetaext}
\Theta _\nu  (z) = \frac{1}{2\im}\log \left( \e^{ - \frac{\pi }{2}\im} \frac{K_\nu  \left( z\e^{ - \frac{\pi }{2}\im}  \right)}{K_\nu  \left( z\e^{\frac{\pi }{2}\im} \right)} \right)
\end{equation}
for $-\frac{3}{2} < \nu < \frac{3}{2}$ and $|\arg z|\leq \frac{\pi}{2}$. Note that the modified Bessel function of the second kind is an even function of its order $\nu$ and so is $\Theta _\nu  (z)$.

\begin{lemma}\label{lemma2} Assume that $-\frac{3}{2} < \nu < \frac{3}{2}$. Then for any $s>0$,
\begin{equation}\label{retheta}
\Re \Theta _\nu  \left( s\e^{\frac{\pi }{2}\im} \right) \begin{cases} <0 & \text{ if } \; -\frac{1}{2} < \nu < \frac{1}{2}, \\ =0 & \text{ if } \; \nu =\pm \frac{1}{2}, \\ >0 & \text{ if } \; \frac{1}{2} < \pm\nu < \frac{3}{2}. \end{cases}
\end{equation}
Furthermore,
\begin{equation}\label{thetaasymp}
\Theta _\nu  (z) =  z + \O\left( \frac{1}{\left| z \right|} \right)
\end{equation}
as $|z| \to +\infty$ in the closed sector $|\arg z|\leq \frac{\pi}{2}$, and
\begin{equation}\label{rethetaasymp}
\Re \Theta _\nu  \left( s\e^{\frac{\pi }{2}\im} \right) =  - \frac{\cos (\pi \nu )}{2}\e^{ - 2s} (1+o(1))
\end{equation}
as $s\to +\infty$.  
\end{lemma}

\begin{proof} Since $\Theta_\nu (z)$ is an even function of $\nu$, it is enough to consider the case $0 \leq \nu < \frac{3}{2}$. From \eqref{Thetaext} and the known analytic continuation formula for the modified Bessel function of the second kind \cite[\href{http://dlmf.nist.gov/10.34.E2}{Eq. 10.34.2}]{DLMF}, we deduce
\begin{gather}\label{thetaonim}
\begin{split}
\Theta _\nu  \left( s\e^{\frac{\pi }{2}\im} \right) & = \frac{1}{2\im}\log \left( \e^{ - \frac{\pi}{2} \im} \frac{K_\nu  (s)}{K_\nu  (s\e^{\pi \im} )} \right) = \frac{1}{2\im}\log \left( \e^{ - \frac{\pi}{2} \im} \frac{K_\nu  (s)}{\e^{ - \pi \im\nu } K_\nu  (s) - \pi \im I_\nu  (s)} \right)
\\ & = - \frac{1}{2\im}\log \left( \e^{ - \pi \im\left( \nu  - \frac{1}{2} \right)}  + \pi \frac{I_\nu  (s)}{K_\nu  (s)} \right) \\ & =  - \frac{1}{2\im}\log \left( \sin (\pi \nu ) + \pi \frac{I_\nu  (s)}{K_\nu  (s)} + \im\cos (\pi \nu ) \right)
\end{split}
\end{gather}
for all $s>0$. Here, $I_\nu(s)$ denotes the modified Bessel function of the first kind \cite[\href{http://dlmf.nist.gov/10.25.ii}{\S10.25(ii)}]{DLMF}. Consequently,
\begin{equation}\label{rethetaformula}
\Re \Theta _\nu  \left( s\e^{\frac{\pi }{2}\im} \right) =  - \frac{1}{2}\arg \left( \sin (\pi \nu ) + \pi \frac{I_\nu  (s)}{K_\nu  (s)} + \im\cos (\pi \nu ) \right).
\end{equation}
The desired result \eqref{retheta} now follows readily, since $\frac{I_\nu  (s)}{K_\nu  (s)}$ is real and positive for all $0 \leq \nu < \frac{3}{2}$ and $s>0$ \cite[\href{http://dlmf.nist.gov/10.37}{\S10.37}]{DLMF}.

We proceed with the proof of \eqref{thetaasymp}. It is well known that
\[
K_\nu  (z) = \sqrt {\frac{\pi }{2z}} \e^{ - z} \left( 1 + \O\left( \frac{1}{\left| z \right|}\right) \right)
\]
as $|z|\to +\infty$ in the closed sector $\left| \arg z \right| \le \pi$ (see, for instance, \cite[\href{http://dlmf.nist.gov/10.40.E2}{Eq. 10.40.2}]{DLMF}). Combining this asymptotic formula with \eqref{Thetaext} and simplifying shows that
\[
\Theta _\nu  (z) = z + \log \left( 1 + \O\left( \frac{1}{\left| z \right|} \right) \right) = z + \O\left( \frac{1}{\left| z \right|} \right)
\]
as $|z|\to +\infty$ in $\left| \arg z \right| \le \frac{\pi}{2}$.

Finally, the leading order asymptotics of the modified Bessel functions \cite[\href{http://dlmf.nist.gov/10.40.i}{\S10.40(i)}]{DLMF}, combined with \eqref{rethetaformula}, give
\[
\Re \Theta _\nu  \left( s\e^{\frac{\pi }{2}\im} \right) =  - \frac{1}{2}\arctan \left( \frac{\cos (\pi \nu )}{\sin (\pi \nu ) + \e^{2s}(1+o(1)) } \right) =  - \frac{\cos (\pi \nu )}{2}\e^{ - 2s} (1+o(1))
\]
as $s\to +\infty$.    
\end{proof}

\begin{figure}[!ht]
	\centering
		\includegraphics[width=0.26\textwidth]{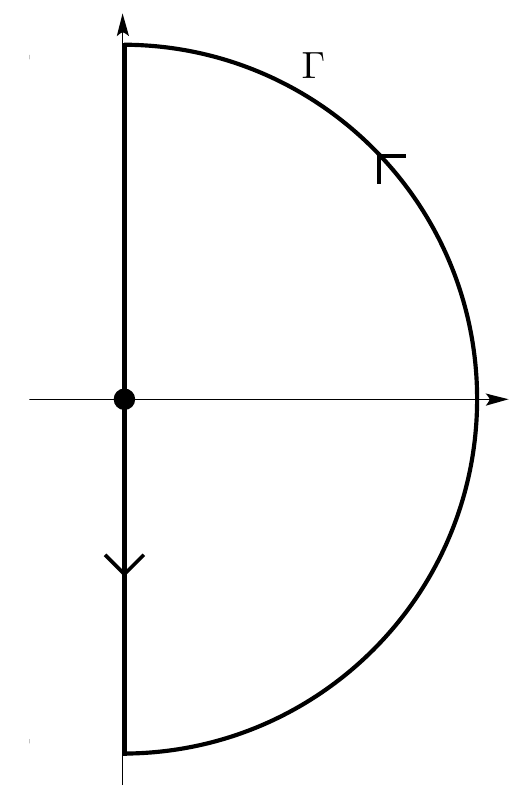}
		\caption{The contour of integration $\Gamma$.}
		\label{fig1}
\end{figure}

\begin{proof}[Proof of Theorem \ref{thmtheta}] Let us first consider the case that $-\frac{3}{2}< \nu < \frac{3}{2}$. Let $\Gamma$ be the $D$-shaped, positively oriented closed contour depicted in Figure \ref{fig1}. By Cauchy's integral theorem in its most general form \cite[Ch. 2, \S2.3, Theorem 2.4]{Kodaira2007}, we can write
\[
\theta _\nu  (z) - z + \left( \tfrac{1}{2}\nu  + \tfrac{1}{4} \right)\pi  = \Theta_\nu  (z) -z= \frac{1}{2\pi \im}\oint_\Gamma \frac{\Theta _\nu  (t)-t}{t - z}\d t  - \underbrace {\frac{1}{2\pi \im}\oint_\Gamma \frac{\Theta _\nu  (t) -t}{t + z}\d t }_0
\]
for all $z>0$ lying inside the contour $\Gamma$. From \eqref{thetaasymp} it is seen that when the radius of the semicircular portion of the contour $\Gamma$ tends to $+\infty$, the integrals along it tend to zero. Accordingly,
\begin{align*}
\theta _\nu  (z) - z + \left( \tfrac{1}{2}\nu  + \tfrac{1}{4} \right)\pi  =\; & \frac{1}{2\pi \im }\int_{ + \im \infty }^0 \frac{\Theta _\nu  (t) -t}{t - z}\d t  + \frac{1}{2\pi \im }\int_0^{ - \im \infty } \frac{\Theta _\nu  (t) -t}{t - z}\d t
\\ & - \frac{1}{2\pi \im }\int_{ + \im \infty }^0 \frac{\Theta _\nu  (t) -t}{t + z}\d t  - \frac{1}{2\pi \im }\int_0^{ - \im \infty } \frac{\Theta _\nu  (t) -t}{t + z}\d t
\\ =\; &  - \frac{1}{2\pi }\int_0^{ + \infty } \frac{\Theta _\nu  \left( s\e^{\frac{\pi}{2}\im } \right) -\im s}{\im s - z}\d s  + \frac{1}{2\pi}\int_0^{ + \infty } \frac{\Theta _\nu  \left( s\e^{ - \frac{\pi}{2}\im } \right) +\im s}{\im s + z}\d s 
\\ & + \frac{1}{2\pi}\int_0^{ + \infty } \frac{\Theta _\nu  \left( s\e^{\frac{\pi}{2}\im }\right)-\im s}{\im s + z}\d t  - \frac{1}{2\pi}\int_0^{ + \infty } \frac{\Theta _\nu  \left( s\e^{ - \frac{\pi}{2}\im } \right) +\im s}{\im s - z}\d s 
\\ =\; & \frac{1}{z}\frac{2}{\pi}\int_0^{ + \infty } \frac{\Re \Theta _\nu  \left( s\e^{\frac{\pi}{2}\im} \right)}{1 + (s/z)^2}\d s .
\end{align*}
In the last step, use has been made of the fact that
\[
2\Re \Theta _\nu  \left( s \e^{\frac{\pi }{2}\im } \right) = \Theta _\nu  \left( s \e^{\frac{\pi }{2}\im } \right) + \Theta _\nu  \left( s \e^{ - \frac{\pi }{2}\im } \right)
\]
for all $s>0$, a consequence of the Schwarz reflection principle. The restriction on $z$ can now be removed by an appeal to analytic continuation, and thus,
\begin{equation}\label{thetaexact}
\theta _\nu  (z) = z - \left( \tfrac{1}{2}\nu  + \tfrac{1}{4} \right)\pi + \frac{1}{z}\frac{2}{\pi}\int_0^{ + \infty } \frac{\Re \Theta _\nu  \left( s\e^{\frac{\pi}{2}\im} \right)}{1 + (s/z)^2}\d s
\end{equation}
provided $\Re z>0$. Note that, by \eqref{rethetaasymp}, $\Re \Theta _\nu  \left( s\e^{\frac{\pi}{2}\im} \right)$ decays faster than any negative power of $s$ as $s\to +\infty$. Therefore, for any positive integer $N$, $\Re z>0$ and $s>0$, we can expand the integrand in \eqref{thetaexact} using
\begin{equation}\label{geometric}
\frac{1}{1 + (s/z)^2 } = \sum\limits_{n = 1}^{N - 1} \frac{1}{z^{2n - 2}}( - 1)^{n - 1} s^{2n - 2}  + \frac{1}{z^{2N - 2}}( - 1)^{N - 1} \frac{s^{2N - 2} }{1 + (s/z)^2 },
\end{equation}
to obtain the truncated expansion \eqref{thetaexpexact} with
\begin{equation}\label{tcoeffint}
t_n (\nu ) = ( - 1)^{n - 1} \frac{2}{\pi}\int_0^{ + \infty } s^{2n - 2} \Re \Theta _\nu  \left( s\e^{\frac{\pi }{2}\im}  \right)\d s 
\end{equation}
and
\begin{equation}\label{remainderformula}
R_N^{(\theta )} (\nu ,z) = \frac{1}{z^{2N - 1}}( - 1)^{N - 1} \frac{2}{\pi}\int_0^{ + \infty } \frac{s^{2N - 2} \Re \Theta _\nu  \left( s\e^{\frac{\pi }{2}\im}  \right)}{1 + (s/z)^2 }\d s .
\end{equation}
Formula \eqref{tcoeffint}, together with \eqref{retheta}, implies that
\[
( - 1)^n t_n (\nu ) \begin{cases} >0 & \text{ if } \; -\frac{1}{2} < \nu < \frac{1}{2}, \\ =0 & \text{ if } \; \nu =\pm\frac{1}{2}, \\ <0 & \text{ if } \; \frac{1}{2} < \pm\nu < \frac{3}{2}. \end{cases}
\]
The fact that the $t_n(\nu)$'s are polynomials in $\nu^2$ of degree $n$ is shown in Appendix \ref{Appendix}. Hence, it remains to prove the claimed bounds on the remainder term $R_N^{(\theta )} (\nu ,z)$. To this end, we combine the elementary inequality
\[
\left| 1 + u^2 \right| \geq \begin{cases} 1 & \text{ if } \; \left|\arg u\right| \leq \frac{\pi}{4}, \\ |\sin (2\arg u)| & \text{ if } \; \frac{\pi}{4} < \left|\arg u\right| < \frac{\pi}{2}, \end{cases}
\]
the fact that $\Re \Theta _\nu  (s\e^{\frac{\pi }{2}\im} )$ does not change sign for $s>0$ and fixed $\nu$ (Lemma \ref{lemma2}), and formula \eqref{tcoeffint} (with $N$ in place of $n$), and deduce
\begin{align*}
\left| R_N^{(\theta )} (\nu ,z) \right| & \le \frac{1}{\left| z \right|^{2N - 1} }\frac{2}{\pi }\int_0^{ + \infty } \frac{s^{2N - 2} \left| \Re \Theta _\nu  (s \e^{\frac{\pi }{2}\im} ) \right|}{\left| 1 + (s/z)^2  \right|}\d s  = \frac{1}{\left| z \right|^{2N - 1} }\left| \frac{2}{\pi }\int_0^{ + \infty } \frac{s^{2N - 2} \Re \Theta _\nu  (s\e^{\frac{\pi }{2}\im} )}{\left| 1 + (s/z)^2 \right|}\d s  \right|
\\ & \le \frac{1}{\left| z \right|^{2N - 1} }\left| \frac{2}{\pi }\int_0^{ + \infty } s^{2N - 2} \Re \Theta _\nu  (s\e^{\frac{\pi }{2}\im} )\d s \right| \times \begin{cases} 1 & \text{ if } \; \left|\arg z\right| \leq \frac{\pi}{4}, \\ |\csc (2\arg z)| & \text{ if } \; \frac{\pi}{4} < \left|\arg z\right| < \frac{\pi}{2} \end{cases} \\ & = \frac{\left| t_N (\nu ) \right|}{\left| z \right|^{2N - 1} } \times \begin{cases} 1 & \text{ if } \; \left|\arg z\right| \leq \frac{\pi}{4}, \\ |\csc (2\arg z)| & \text{ if } \; \frac{\pi}{4} < \left|\arg z\right| < \frac{\pi}{2} . \end{cases}
\end{align*}
It follows readily from \eqref{retheta}, \eqref{tcoeffint} and \eqref{remainderformula} that, when $z>0$, $R_N^{(\theta )} (\nu ,z)$ and $\frac{t_N (\nu )}{z^{2N - 1} }$ have the same sign.

The cases when $\nu =\pm \frac{3}{2}$ may be proved using \eqref{32case} and the integral representation
\[
\arctan \left( \frac{1}{z} \right) = \frac{1}{z}\int_0^1 \frac{\d s}{1 + (s/z)^2 }, \quad \Re z>0.
\]
We leave the details to the interested reader.
\end{proof}

\section{Proof of Theorem \ref{thmX}}\label{Section4}

To prove Theorem \ref{thmX}, we require some lemmas.

\begin{lemma} Suppose that $-\frac{1}{2}<\nu <\frac{1}{2}$. Then the function $w=\Theta_\nu(z)$, $\Re z>0$, is injective and its range contains the closed right half-plane $\Re w \geq 0$.
\end{lemma}

\begin{proof} Since $\Theta_\nu(z)$ is an even function of $\nu$, it suffices to consider the case $0 \leq \nu < \frac{1}{2}$. As $\sin(\pi \nu)>0$ for $0 \leq \nu <\frac{1}{2}$, we find from \eqref{thetaonim} that
\[
\Re \Theta _\nu  \left(s\e^{\frac{\pi }{2}\im} \right) =  - \frac{1}{2}\arctan \left( \frac{\cos (\pi \nu )}{\sin (\pi \nu ) + \pi \frac{I_\nu  (s)}{K_\nu  (s)}} \right) <0
\]
and
\[
\Im \Theta _\nu  \left(s\e^{\frac{\pi }{2}\im} \right) = \frac{1}{2}\log \left( 1 + 2\pi \sin (\pi \nu )\frac{I_\nu  (s)}{K_\nu  (s)} + \pi ^2 \frac{I_\nu ^2 (s)}{K_\nu ^2 (s)} \right)>0
\]
for all $s>0$. Now $\frac{I_\nu  (s)}{K_\nu  (s)}$, as a function of $s>0$, is positive and strictly monotonically increasing (cf. \cite[\href{http://dlmf.nist.gov/10.37}{\S10.37}]{DLMF}). By the Schwarz reflection principle,
\[
\Theta _\nu  \left(s\e^{ - \frac{\pi }{2}\im} \right) = \Re \Theta _\nu  \left(s\e^{\frac{\pi }{2}\im} \right) - \im\Im \Theta _\nu  \left(s\e^{\frac{\pi }{2}\im} \right)
\]
for any $s>0$. Therefore, the function $\Theta_\nu(z)$ is injective on the imaginary axis, which it maps into an infinite curve $\mathcal{G}_\nu$ that lies entirely in the left half-plane and is symmetric with respect to the real axis. The curve $\mathcal{G}_\nu$ crosses the negative real axis at the point
\[
\mathop {\lim }\limits_{s \to 0 + } \Theta _\nu  \left( s\e^{ \pm \frac{\pi }{2}\im} \right) = \mathop {\lim }\limits_{s \to 0 + } \theta _\nu  \left( s \e^{ \pm \frac{\pi }{2}\im}  \right) + \left( \tfrac{1}{2}\nu  + \tfrac{1}{4} \right)\pi  = \left( \tfrac{1}{2}\nu  - \tfrac{1}{4} \right)\pi.
\]

To proceed further, we require that $\Theta_\nu(z)$ is not just continuous, but also analytic on the imaginary axis with the exception of the origin. For this, we note that when $0 \leq \nu < \frac{1}{2}$, the Hankel functions $H_\nu ^{(1)} (z)$ and $H_\nu ^{(2)} (z)$ do not vanish in a vicinity of the rays $\arg z=\pm\frac{\pi}{2}$ \cite{Cruz1982}, i.e., $K_\nu (z)$ is non-zero near $\arg z =\pm \pi$. Therefore, \eqref{Thetaext} defines a holomorphic function of $z$ in a neighbourhood of the rays $\arg z =\pm\frac{\pi}{2}$. The mapping
\[
\zeta  = \zeta (z) = \frac{1}{1 + z}
\]
constitutes a bijection between the imaginary axis and the circle
\[
\mathcal{C}=\left\{ \zeta :\left| \zeta  - \tfrac{1}{2} \right| = \tfrac{1}{2} \right\},
\]
and between the open right half-plane and the interior $I(\mathcal{C})$ of $\mathcal{C}$. Define the function
\[
\Theta _\nu^\ast (\zeta ) = \Theta_\nu \big( \tfrac{1}{\zeta } - 1  \big)
\]
acting on the closure of $I(\mathcal{C})$. Observe that $\Theta _\nu^\ast (\zeta )$ is analytic on the closure of $I(\mathcal{C})$ except at the points $\zeta =0$ and $1$. It remains continuous at $\zeta = 1$, and at $\zeta =0$ it is of infinitely large order $1$:
\[
\mathop {\lim }\limits_{\zeta  \to 0} \zeta \Theta_\nu ^\ast  (\zeta ) = 1
\]
(this follows from the asymptotics \eqref{thetaasymp} of $\Theta _\nu(z)$). Since $\Theta_\nu(z)$ is one-to-one on the imaginary axis, $\Theta_\nu^\ast (\zeta )$ is one-to-one on $\mathcal{C}$. Thus, by a general version of the boundary correspondence principle \cite[Ch. 4, \S18, Theorem 4.9]{Markushevich1965}, $\Theta _\nu^\ast (\zeta)$ is injective on the interior of $\mathcal{C}$. The image of the interior is one of the domains with boundary $\mathcal{G}_\nu$, i.e., the domain on the left of an observer moving with the point $w=\Theta_\nu^\ast (\zeta )$ as $\zeta$ traverses $\mathcal{C}$ once in the positive direction. By the properties of $\mathcal{G}_\nu$, this domain includes the closed half-plane $\Re w \geq 0$ (see Figure \ref{fig2}).
\end{proof}

\begin{figure}[!ht]
	\centering
		\includegraphics[width=0.4\textwidth]{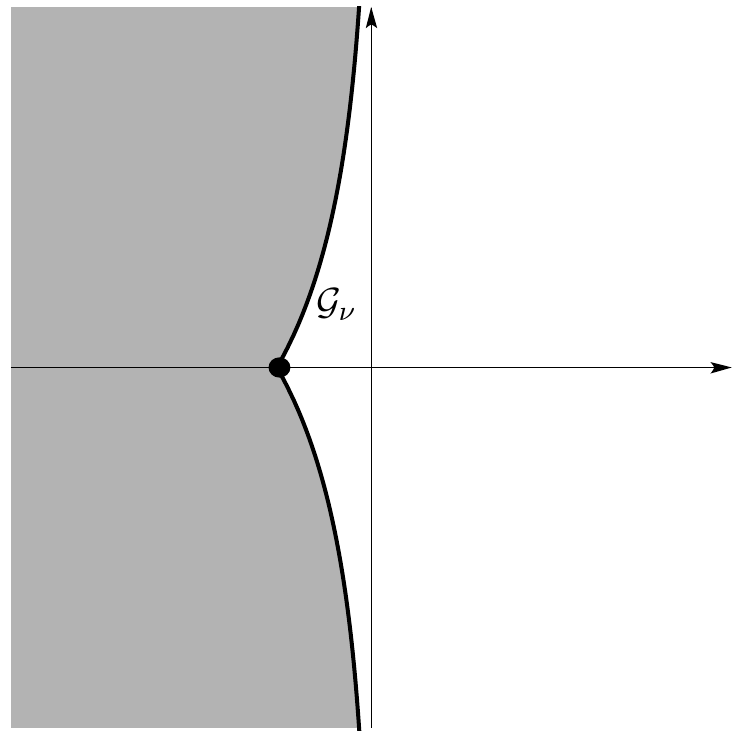}
		\caption{The range of the function $w=\Theta_\nu (z)$ (unshaded) bounded by the curve $\mathcal{G}_\nu$ which intersects the real axis at $w=\left( \tfrac{1}{2}|\nu|  - \tfrac{1}{4} \right)\pi<0$.}
		\label{fig2}
\end{figure}

Since $\Theta_\nu (z)$ is holomorphic and injective on the domain $\Re z>0$, it has a holomorphic inverse defined on the range of $\Theta_\nu (z)$. Note that, due to the way we defined $\Theta_\nu (z)$ and $X_\nu(w)$, this inverse function must coincide with $X_\nu(w)$ when $w$ is real and $w > \left( \frac{1}{2}|\nu|  - \frac{1}{4} \right)\pi$. Thus, $z = X_\nu(w)$ extends analytically to a domain which includes the closed half-plane $\Re w \geq 0$ and its range is $\Re z>0$.

\begin{lemma}\label{lemma6} Suppose that $-\frac{1}{2} < \nu < \frac{1}{2}$. Then for any $s>0$,
\begin{equation}\label{Xpositive}
\Re X_\nu ( \im s ) >0.
\end{equation}
Furthermore,
\begin{equation}\label{Xleadingasymptotic}
X_\nu  (w) = w+\O\left(\frac{1}{|w|}\right)
\end{equation}
as $|w| \to +\infty$ in the closed sector $|\arg w|\leq \frac{\pi}{2}$, and
\begin{equation}\label{ReXasymp}
\Re X_\nu ( \im s ) = o(s^{-r})
\end{equation}
as $s\to +\infty$, with any $r>0$.
\end{lemma}

\begin{proof} Since the domain of the function $X_\nu (w)$ contains the imaginary axis and its range is the right half-plane, the inequality \eqref{Xpositive} follows readily.

To prove the asymptotic formulae, we begin by showing that
\begin{equation}\label{Xleadingasymp}
X_\nu (w) = w(1 + o(1))
\end{equation}
as $|w| \to +\infty$ in the sector $|\arg w|\leq \frac{\pi}{2}$. By the properties of the function $\Theta_\nu(z)$, for each $w$ in the sector $|\arg w|\leq \frac{\pi}{2}$, there is a unique $z$, $\Re z>0$ such that $w=\Theta_\nu(z)$. Thus, it is enough to prove that
\[
\frac{z}{\Theta _\nu  (z)} \to 1
\]
as $|z| \to +\infty$ in $|\arg z|\leq \frac{\pi}{2}$. But this is a simple consequence of \eqref{thetaasymp}. Now, by \eqref{Thetadef} and Theorem \ref{thmtheta},
\begin{equation}\label{Thetaasymptoticseries}
\Theta _\nu  (z) \sim z + \sum\limits_{n = 1}^\infty \frac{t_n (\nu )}{z^{2n-1}}
\end{equation}
as $|z| \to +\infty$ in any closed sub-sector of $|\arg z|\leq \frac{\pi}{2}$. To proceed further, we require that this asymptotic expansion holds in $|\arg z|\leq \frac{\pi}{2}$ itself. This can be shown, for example, by substituting the large-$z$ asymptotic series \cite[\href{http://dlmf.nist.gov/10.40.E2}{Eq. 10.40.2}]{DLMF}, which is valid when $|\arg z|\leq \pi$, of the function $K_\nu(z)$ into \eqref{Thetaext} and expanding the right-hand side in inverse powers of $z$. Alternatively, we can refer directly to a result of Luke \cite[Ch. V, \S5.11.4, Eqs. (9) and (11)]{Luke1969}. Let $z$, $\Re z>0$ be the unique point that corresponds to $w$, $|\arg w|\leq \frac{\pi}{2}$ via the relation $w=\Theta_\nu(z)$. From \eqref{Xleadingasymp}, we can assert that
\[
z = X_\nu  (w) = w(1 + o(1))
\]
as $|w| \to +\infty$ in the sector $|\arg w|\leq \frac{\pi}{2}$. Beginning with this approximation and repeatedly re-substituting in the right-hand side of 
\[
X_\nu  (w) = w - \frac{t_1 (\nu )}{z} - \frac{t_2 (\nu )}{z^3} -  \cdots  - \frac{t_{N - 1} (\nu )}{z^{2N - 3}} + \O\left( \frac{1}{|z|^{2N - 1} } \right),
\]
$N \geq 2$ being an arbitrary integer, we see that there exists a representation of the form 
\begin{equation}\label{Xexp}
X_\nu  (w) = w - \frac{d_1 (\nu )}{w} - \frac{d_2 (\nu )}{w^2 } -  \cdots  - \frac{d_{N - 1} (\nu )}{w^{N - 1}} + \O\left( \frac{1}{|w|^N } \right)
\end{equation}
as $|w| \to +\infty$ in $|\arg w|\leq \frac{\pi}{2}$, where the coefficients $d_n(\nu )$ are polynomials in the $t_n(\nu )$'s which are independent of $N$ and $w$. For large positive $w$, \eqref{Xexp} must coincide with McMahon's expansion, whence, the uniqueness theorem of asymptotic power series \cite[Ch. 1, \S7.2]{Olver1997} tells us that
\[
d_{2n - 1} (\nu ) = -c_n (\nu )\quad \text{and} \quad d_{2n} (\nu ) = 0
\]
for all positive integer $n$. Accordingly,
\begin{equation}\label{Xasymptoticseries}
X_\nu  (w) \sim w + \sum\limits_{n = 1}^\infty \frac{c_n (\nu )}{w^{2n - 1}}
\end{equation}
as $|w| \to +\infty$ in $|\arg w|\leq \frac{\pi}{2}$. This asymptotic expansion immediately implies the weaker result \eqref{Xleadingasymptotic}. It also follows that for any $N\geq 1$,
\[
\Re X_\nu  (\im s) = \Re \left( X_\nu  (\im s) - \im s - \sum\limits_{n = 1}^{N - 1} \frac{c_n (\nu )}{(\im s)^{2n - 1} } \right) = \Re \O\left( \frac{1}{s^{2N - 1} } \right) = \O\left( \frac{1}{s^{2N - 1} } \right)
\]
as $s\to +\infty$, confirming the claim \eqref{ReXasymp}.
\end{proof}

\paragraph{\textbf{Remark.}} With the aid of more sophisticated tools, such as \'Ecalle's alien derivatives, it is possible to obtain more precise asymptotics like
\[
\Re X_\nu ( \im s ) = \frac{\cos(\pi \nu)}{2}\e^{-2s}(1+o(1)) 
\]
as $s\to +\infty$, an analogue of \eqref{rethetaasymp} (In\^es Aniceto, personal communication, May 2020). Nevertheless, the weaker result \eqref{ReXasymp} is adequate for our purposes.

\begin{proof}[Proof of Theorem \ref{thmX}] The proof is completely analogous to the proof of Theorem \ref{thmtheta}, therefore we present only the resulting integral formulae which could be of independent interest. Following the initial steps in the proof of Theorem \ref{thmtheta} and referring to Lemma \ref{lemma6}, we derive
\[
X_\nu  (w) = w + \frac{1}{w}\frac{2}{\pi}\int_0^{ + \infty } \frac{\Re X_\nu  (\im s)}{1 + (s/w)^2 }\d s 
\]
for $-\frac{1}{2}<\nu<\frac{1}{2}$ and $\Re w>0$. By expanding the integrand via \eqref{geometric} (with $w$ in place of $z$), we obtain the truncated expansion \eqref{Xtruncexp} with
\begin{equation}\label{cnformula}
c_n (\nu ) = ( - 1)^{n - 1} \frac{2}{\pi}\int_0^{ + \infty } s^{2n - 2} \Re X_\nu  (\im s)\d s 
\end{equation}
and
\begin{equation}\label{RXrep}
R_N^{(X)} (\nu ,w) = \frac{1}{w^{2N - 1}}( - 1)^{N - 1} \frac{2}{\pi}\int_0^{ + \infty } \frac{s^{2N - 2} \Re X_\nu  (\im s)}{1 + (s/w)^2 }\d s .
\end{equation}
From \eqref{Xpositive} and \eqref{cnformula}, we can assert that $(-1)^n c_n (\nu )<0$. It is proved in Appendix \ref{Appendix} that the $c_n(\nu)$'s are polynomials in $\nu^2$ of degree $n$. To estimate the remainder $R_N^{(X)} (\nu ,w)$, we can proceed as in the case of $R_N^{(\theta)} (\nu ,z)$.
\end{proof}

\section{Proof of Theorem \ref{thmT}}\label{Section5}

We begin by observing that $T(w)$, defined via \eqref{Tdef}, extends analytically to the closed half-plane $\Re w \geq 0$. Indeed, $z=X_\nu(w)$ extends analytically to a domain which includes the closed half-plane $\Re w \geq 0$ and its range is $\Re z>0$. In the following lemma, we establish some additional properties of $T(w)$ that will be used to prove Theorem \ref{thmT}.

\begin{lemma}\label{lemma7} For any $s>0$,
\begin{equation}\label{imTineq}
\Im \left(  \e^{-\frac{\pi}{3}\im} T ( \im s ) \right) < 0.
\end{equation}
Furthermore,
\begin{equation}\label{Tgrowth}
w^{ - 2/3} T(w) = 1 + \O\left( \frac{1}{\left| w \right|^2 } \right)
\end{equation}
as $|w| \to +\infty$ in the closed sector $|\arg w|\leq \frac{\pi}{2}$, and
\begin{equation}\label{imTasymp}
\Im \left(  \e^{-\frac{\pi}{3}\im} T ( \im s ) \right) = o(s^{ - r} )
\end{equation}
as $s\to +\infty$, with any $r>0$.
\end{lemma}

\begin{proof} If $s>0$, then, by \eqref{Tdef},
\begin{equation}\label{argT}
\arg \left(  \e^{-\frac{\pi}{3}\im} T ( \im s ) \right) = \arg \left( T ( \im s ) \right) - \tfrac{\pi}{3} = \tfrac{2}{3}\arg \left( X_{1/3} \left( \tfrac{2}{3}\im s \right) \right) - \tfrac{\pi }{3}.
\end{equation}
Since the domain of the function $X_{1/3}(w)$ contains the imaginary axis and its range is the right half-plane, we can assert that
\[
 - \tfrac{2\pi}{3} < \tfrac{2}{3}\arg \left( X_{1/3} \left( \tfrac{2}{3}\im s \right) \right) - \tfrac{\pi}{3} < 0,
\]
which implies the inequality \eqref{imTineq}.

The asymptotic formula \eqref{Tgrowth} is a simple consequence of the definition \eqref{Tdef} and the estimate \eqref{Xleadingasymptotic}. Indeed,
\begin{align*}
w^{ - 2/3} T(w) & = \left( w^{ - 1} \tfrac{3}{2}X_{1/3} \left( \tfrac{2}{3}w \right) \right)^{2/3}  = \left( w^{ - 1} \frac{3}{2}\left( \frac{2}{3}w + \O\left( \frac{1}{\left| w \right|} \right) \right) \right)^{2/3} \\ & = \left( 1 + \O\left( \frac{1}{\left| w \right|^2 } \right) \right)^{2/3}  = 1 + \O\left( \frac{1}{\left| w \right|^2 } \right)
\end{align*}
as $|w| \to +\infty$ in the closed sector $|\arg w|\leq \frac{\pi}{2}$.

Finally, to establish \eqref{imTasymp}, we may proceed as follows. From Lemma \ref{lemma6}, one can infer that
\begin{align*}
\arg \left( X_{1/3} \left( \tfrac{2}{3}\im s \right) \right) & = \arctan \left( \frac{\Im X_{1/3} \left( \frac{2}{3}\im s \right)}{\Re X_{1/3} \left( \frac{2}{3}\im s \right)} \right) = \frac{\pi }{2} - \arctan \left( \frac{\Re X_{1/3} \left( \frac{2}{3}\im s \right)}{\Im X_{1/3} \left( \frac{2}{3}\im s \right)} \right)  \\ & = \frac{\pi}{2} - \arctan \left( \frac{o(s^{ - r} )}{\frac{2}{3}s + \O\left( \frac{1}{s} \right)} \right) = \frac{\pi}{2} + o(s^{ - r-1} )
\end{align*}
as $s\to +\infty$, with any $r>0$. Therefore, by \eqref{argT},
\[
\arg \left(  \e^{-\frac{\pi}{3}\im} T ( \im s ) \right) = o(s^{-r-1}).
\]
Additionally, from \eqref{Tgrowth},
\[
\left| \e^{-\frac{\pi}{3}\im} T ( \im s ) \right| = s^{2/3} + o(1).
\]
Taking these two estimates into account, we conclude that
\begin{align*}
\Im \left( \e^{-\frac{\pi}{3}\im} T ( \im s ) \right) & = \left|\e^{-\frac{\pi}{3}\im} T ( \im s )\right|\sin \left( \arg \left( \e^{-\frac{\pi}{3}\im} T ( \im s ) \right) \right) \\ & = \left(s^{2/3}+o(1)\right)\sin (o(s^{ - r-1} )) = o(s^{ - r} )
\end{align*}
as $s\to +\infty$, with any $r>0$.
\end{proof}

\begin{figure}[!ht]
	\centering
		\includegraphics[width=0.26\textwidth]{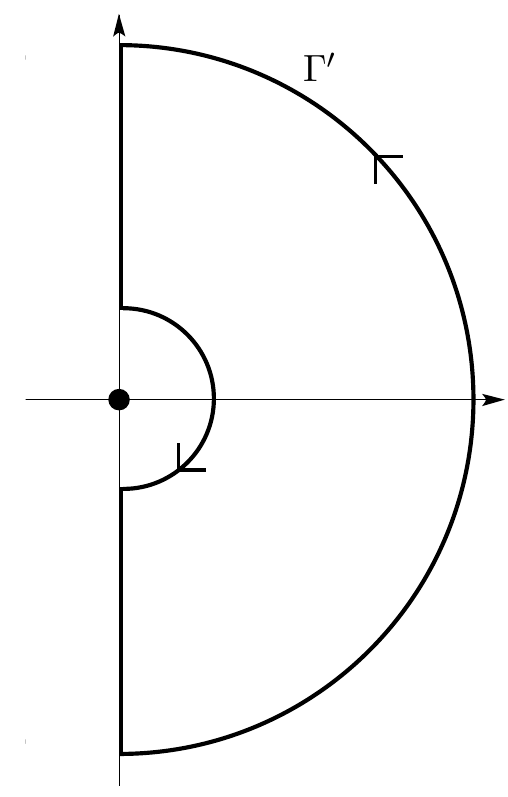}
		\caption{The contour of integration $\Gamma'$.}
		\label{fig3}
\end{figure}

\begin{proof}[Proof of Theorem \ref{thmT}] Let $\Gamma'$ be the positively oriented closed curve shown in Figure \ref{fig3}. By Cauchy's integral theorem,
\[
w^{ - 2/3} T(w) - 1 = \frac{1}{2\pi \im}\oint_{\Gamma '} \frac{t^{ - 2/3} T(t) - 1}{t - w}\d t  + \underbrace {\frac{1}{2\pi \im}\oint_{\Gamma '} \frac{t^{ - 2/3} T(t) - 1}{t + w}\d t }_0
\]
for all $w > 0$ enclosed by the contour $\Gamma'$. From \eqref{Tgrowth} it is seen that when the radius of the large semicircular
portion of the contour $\Gamma'$ approaches $+\infty$, the integrals along it tend to zero. Similarly, since $T(t)$ is bounded at the origin, when the radius of the small semicircular arc tends to $0$, the integrals along it tend to zero. Hence,
\begin{align*}
w^{ - 2/3} T(w) - 1 =\; & \frac{1}{2\pi \im}\int_{ + \im\infty }^0 \frac{t^{ - 2/3} T(t) - 1}{t - w}\d t  + \frac{1}{2\pi \im}\int_0^{ - \im\infty } \frac{t^{ - 2/3} T(t) - 1}{t - w}\d t
\\ & + \frac{1}{2\pi \im}\int_{ + \im\infty }^0 \frac{t^{ - 2/3} T(t) - 1}{t + w}\d t  + \frac{1}{2\pi \im}\int_0^{ - \im\infty } \frac{t^{ - 2/3} T(t) - 1}{t + w}\d t
\\ =\; &  - \frac{1}{2\pi }\int_0^{ + \infty } \frac{s^{ - 2/3} \e^{-\frac{\pi}{3}\im} T(\im s) - 1}{\im s - w}\d s  + \frac{1}{2\pi}\int_0^{ + \infty } \frac{s^{ - 2/3} \e^{\frac{\pi}{3}\im} T( - \im s) - 1}{\im s + w}\d s
\\ & - \frac{1}{2\pi }\int_0^{ + \infty } \frac{s^{ - 2/3} \e^{-\frac{\pi}{3}\im} T(\im s) - 1}{\im s + w}\d s + \frac{1}{2\pi }\int_0^{ + \infty } \frac{s^{ - 2/3} \e^{\frac{\pi}{3}\im} T( - \im s) - 1}{\im s - w}\d s
\\ =\; &  - \frac{1}{w^2}\frac{2}{\pi }\int_0^{ + \infty } \frac{s^{1/3} \Im\left(\e^{-\frac{\pi}{3}\im} T(\im s) \right)}{1 + (s/w)^2 }\d s,
\end{align*}
where in arriving at the last line, we made use of the identity
\[
2\im \Im \left( \e^{-\frac{\pi}{3}\im} T(\im s) \right) = \e^{-\frac{\pi}{3}\im} T(\im s) -\e^{\frac{\pi}{3}\im} T(-\im s)
\]
for all $s > 0$. This identity is a consequence of the Schwarz reflection principle. We can now remove the restriction on $w$ using analytic continuation, and therefore,
\begin{equation}\label{Texact}
T(w) = w^{2/3}\left(1 - \frac{1}{w^2}\frac{2}{\pi }\int_0^{ + \infty } \frac{s^{1/3} \Im\left(\e^{-\frac{\pi}{3}\im} T(\im s) \right)}{1 + (s/w)^2 }\d s\right)
\end{equation}
provided $\Re w>0$. Note that, by \eqref{imTasymp}, $\Im\left(\e^{-\frac{\pi}{3}\im} T(\im s) \right)$ decays faster than any negative power of $s$ as $s\to +\infty$. Thus, for any positive integer $N$, $\Re w>0$ and $s>0$, we can expand the integrand in \eqref{Texact} using
\eqref{geometric} (with $w$ in place of $z$), to arrive at the truncated expansion \eqref{Texpexact} with
\begin{equation}\label{Tintformula}
T_n  = ( - 1)^n \frac{2}{\pi }\int_0^{ + \infty } s^{2n - 5/3} \Im \left( \e^{ - \frac{\pi }{3}\im} T(\im s) \right)\d s
\end{equation}
and
\[
R_N^{(T)} (w) = \frac{1}{w^{2N}}( - 1)^N \frac{2}{\pi }\int_0^{ + \infty } \frac{s^{2N - 5/3} \Im \left( \e^{ - \frac{\pi }{3}\im} T(\im s) \right)}{1 + (s/w)^2}\d s .
\]
From \eqref{imTineq} and \eqref{Tintformula}, we can assert that $(-1)^n T_n<0$. It is shown in Appendix \ref{Appendix} that the coefficients $T_n$ are rational numbers. To estimate the remainder $R_N^{(T)} (w)$, we can proceed as in the case of $R_N^{(\theta)} (\nu ,z)$. The details are left to the interested reader.
\end{proof}

\section{Discussion}\label{Section6}

We studied the asymptotic expansions of the large real zeros of the cylinder and Airy functions. In particular, we demonstrated the enveloping property of these expansions, and thereby settled two conjectures proposed by Elbert and Laforgia and by Fabijonas and Olver, respectively.

The conjecture on the zeros of the cylinder function $\mathscr{C}_\nu  (z,\alpha )$ was formulated under the assumption that the order $\nu$ is restricted to the interval $-\frac{1}{2} < \nu < \frac{1}{2}$. This requirement was indeed essential in the proof, as otherwise the domain of analyticity of $X_\nu(w)$ may not include the closed half-plane $\Re w\geq 0$, thus preventing us from deriving the necessary and convenient integral representations. Nevertheless, it is a natural question to ask if a statement similar to Theorem \ref{Elberttheorem} (or even Theorem \ref{thmX}) holds for other values of $\nu$. Graphical depiction indicates that $( - 1)^n c_n (\nu ) >0$ holds for all positive integer $n$ and real $\nu$ contained in the intervals $\frac{1}{2} < \pm \nu < \frac{1}{14}\sqrt {217}  = 1.0522 \ldots$ (see Figure \ref{fig4}), suggesting that McMahon's expansion \eqref{McMahonseries} might envelope the zeros $j_{\nu ,\kappa}$ for such values of $\nu$ too. We remark that an estimate for the error term of the approximation
\[
j_{\nu ,k } \approx \beta _{\nu ,k }  - \frac{4\nu ^2  - 1}{8\beta _{\nu ,k } } - \frac{(4\nu ^2  - 1)(28\nu ^2  - 31)}{384\beta _{\nu ,k }^3}
\]
for the positive zeros of $J_\nu(z)$, under the assumptions $\nu > \frac{1}{2}$ and $j_{\nu ,k}  > (2\nu  + 1)(2\nu  + 3)/\pi$, was given by Gatteschi and Giordano \cite{Gatteschi2000}. However, their bound is not related to the first omitted term of McMahon's series.

In addition to Conjecture \ref{Olver}, Fabijonas and Olver \cite{Fabijonas1999} posed analogous conjectures about the negative zeros $a'_k$ and $b'_k$ of the derivatives $\Ai'(z)$ and $\Bi'(z)$, and also about the critical values $\Ai'(a_k )$, $\Bi'(b_k )$, $\Ai(a'_k )$ and $\Bi(b'_k )$, respectively. To tackle, for example, the problem on the zeros $a'_k$ and $b'_k$ via the contour integration method of this paper, one would need to continue analytically the function $X_{2/3}(w)$ to the closed half-plane $\Re w\geq 0$. The domain of analyticity for $X_{2/3}(w)$ we could currently guarantee is smaller than this region.

The asymptotic expansions we discussed in this paper are (generally) divergent. With a more intricate analysis, one can prove that for large $k$, the asymptotic series \eqref{McMahonseries}, for example, has terms that initially decrease in magnitude, attain a minimum and thereafter start to diverge. Optimal truncation of the series (i.e., truncation at, or near, the numerically least term), given by $N \approx \beta _{\nu ,\kappa }$ ($>0$), then yields an approximation whose error is exponentially small in the asymptotic variable $k$. It is possible to gain further exponential accuracy beyond the least term by suitable re-expansion of the remainder term $R_N^{(X)}(\nu,w)$, using its integral representation \eqref{RXrep} (cf. \cite{Bennett2018,Berry1991}). An Euler transformation approach is particularly efficient in this situation (cf. \cite[Section 3]{Boyd1990}).

Finally, we note that one can also study complex zeros of the function $\mathscr{A}(z,\alpha)$. For example, $\Bi(z)$ has an infinite number of complex zeros lying in the sectors $\frac{\pi}{3} <\arg z<\frac{\pi}{2}$ and $-\frac{\pi}{2} <\arg z<-\frac{\pi}{3}$. These zeros are usually denoted by $\beta_k$, in the former sector, and by $\overline{\beta_k}$, in the conjugate sector, arranged in ascending order of absolute value for $k\geq 1$ (see, for instance, \cite[\href{http://dlmf.nist.gov/9.9.i}{\S9.9(i)}]{DLMF} or \cite{Olver1954}). By combining \cite[\href{http://dlmf.nist.gov/9.6.E4}{Eq. 9.6.4} and \href{http://dlmf.nist.gov/10.27.E6}{Eq. 10.27.6}]{DLMF}, \eqref{phaserelation} and \eqref{Jviaphase}, it can be shown that
\begin{equation}\label{Biformula}
\Bi\left(z\e^{ \pm \frac{\pi }{3}\im} \right) = \frac{\sqrt z}{2}M_{1/3} \left( \tfrac{2}{3}z^{3/2} \right)\left( \left( \tfrac{2\sqrt 3}{3} \pm \im \right)\cos \theta _{1/3} \left( \tfrac{2}{3}z^{3/2} \right) - \sin \theta _{1/3} \left( \tfrac{2}{3}z^{3/2} \right) \right)
\end{equation}
for all $z>0$. From \eqref{thetaMdef} and Theorem \ref{thmtheta}, we can infer that the right-hand side is analytic in the sector $\left| \arg z \right| < \frac{\pi}{3}$. Therefore, \eqref{Biformula} is valid for $\left| \arg z \right| < \frac{\pi}{3}$, and thus
\[
\beta _k = \e^{ \frac{\pi}{3}\im } T\left( \tfrac{3}{8}\pi (4k - 1) + \tfrac{3}{4}\im \log 2 \right) \quad \text{and} \quad \overline{\beta _k} = \e^{-\frac{\pi}{3}\im } T\left( \tfrac{3}{8}\pi (4k - 1) - \tfrac{3}{4}\im \log 2 \right)
\]
for all positive integer $k$. Consequently, the rigorous asymptotic description of the complex zeros $\beta _k$ and $\overline{\beta _k}$ follows immediately from Theorem \ref{thmT}. For an extensive discussion on the existence of complex zeros of $\mathscr{A}(z,\alpha)$, we refer the reader to \cite{Gil2014}.

\begin{figure}[!ht]
	\centering
		\includegraphics[width=0.75\textwidth]{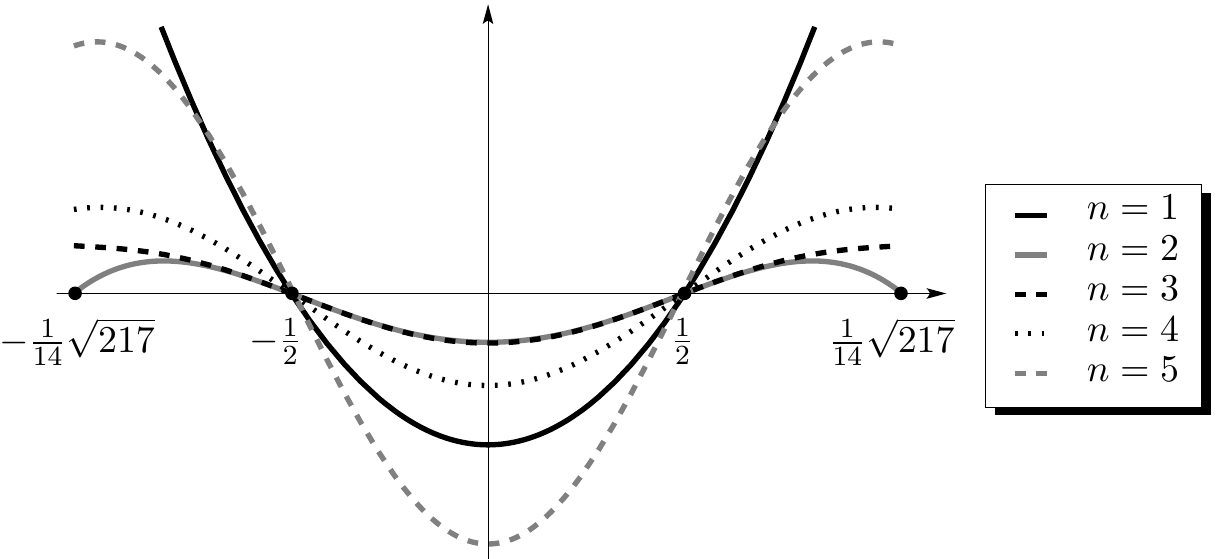}
		\caption{The graphs of the normalised coefficients $\frac{( - 1)^n }{n!}c_n (\nu )$ for $n=1$, $2$, $3$, $4$ and $5$, and $|\nu| \leq \frac{1}{14}\sqrt {217}$. Note that $c_2(\nu)$ has zeros at the endpoints $\nu =\pm \frac{1}{14}\sqrt {217}$.}
		\label{fig4}
\end{figure}

\section*{Acknowledgement} The author was supported by a Premium Postdoctoral Fellowship of the Hungarian Academy of Sciences. The author wish to thank In\^es Aniceto, Christopher J. Howls and Adri B. Olde Daalhuis for useful discussions. The author also thanks the referees for helpful comments and suggestions for improving the presentation.

\appendix

\section{}\label{Appendix}

In this appendix, we show how the various coefficients appearing in this paper can be computed from explicit recursive formulae. Furthermore, we prove that the coefficients $t_n(\nu)$ and $c_n(\nu)$ are polynomials in $\nu^2$ of degree $n$.

To obtain a recurrence relation for the coefficients $t_n(\nu)$, we can exploit the relation \eqref{Mtheta} between $\theta'_\nu (z)$ and $M_\nu(z)$. It is known \cite[\href{http://dlmf.nist.gov/10.18.E17}{Eq. 10.18.17}]{DLMF} that as $z\to +\infty$, with $\nu$ fixed,
\begin{equation}\label{Masymptotic}
\frac{\pi}{2}z M_\nu ^2 (z) \sim 1 + \sum\limits_{n = 1}^\infty \frac{m_n(\nu)}{z^{2n}}
\end{equation}
where
\begin{equation}\label{mformula}
m_n(\nu) = \frac{1}{16^n}\binom{2n}{n}\prod\limits_{k = 1}^n (4\nu ^2  - (2k - 1)^2).
\end{equation}
Since $\theta_\nu(z)$ extends analytically to the half-plane $\Re z>0$, we can differentiate its asymptotic expansion term-wise and obtain
\begin{equation}\label{thetadiffasymptotic}
\theta '_\nu  (z) \sim  1 + \sum\limits_{n = 1}^\infty \frac{(1 - 2n)t_n (\nu )}{z^{2n}}
\end{equation}
as $z\to +\infty$, with $\nu$ being fixed. Substituting \eqref{Masymptotic} and \eqref{thetadiffasymptotic} into \eqref{Mtheta} gives
\[
\left( 1 + \sum\limits_{n = 1}^\infty \frac{(1 - 2n)t_n (\nu )}{z^{2n}} \right) \left(1 + \sum\limits_{n = 1}^\infty \frac{m_n(\nu)}{z^{2n}} \right) \sim 1.
\]
Performing the product of the two asymptotic expansions and equating like powers of $z$, we deduce that
\begin{equation}\label{trecurrence}
t_1 (\nu ) = \frac{4\nu^2-1}{8}, \qquad t_n (\nu ) = \frac{1}{2n - 1}m_n (\nu ) - \sum\limits_{k = 1}^{n - 1} \frac{2k - 1}{2n - 1} t_k (\nu ) m_{n - k}(\nu )
\end{equation}
for all $n\geq 2$ and real $\nu$. A recurrence formula equivalent to \eqref{trecurrence} was also given, up to a slight misprint, in \cite{Heitman2015}. For an alternative recurrence relation, see \cite[Ch. V, \S5.11.4, Eqs. (7), (9) and (11)]{Luke1969}. Since $m_n(\nu)$ is a polynomial in $\nu^2$ of degree $n$, it follows readily by induction that $t_n(\nu)$ is a polynomial in $\nu^2$ of degree at most $n$. The subsequent argument shows that the coefficient $\tau _n$, say, of $\nu^{2n}$ in $t_n (\nu )$ does never vanish. By defining $m_0(\nu)=t_0(\nu)=1$, \eqref{trecurrence} may be re-written as
\begin{equation}\label{trecurrence2}
\sum\limits_{k = 0}^n (1 - 2k)t_k (\nu )m_{n - k} (\nu )  = 0
\end{equation}
for any $n\geq 1$. We can infer from \eqref{mformula} and \eqref{trecurrence2} that
\[
\sum\limits_{k = 0}^n (1 - 2k)\tau _k \frac{1}{4^{n - k} }\binom{2n - 2k}{n - k}  = 0
\]
for all $n\geq 1$. Accordingly, $(1 - 2n)\tau _n$ is the $n$th coefficient in the Maclaurin expansion of $\sqrt{1-z}$, i.e.,
\[
\tau_n = \frac{1}{(2n - 1)^2}\frac{1}{4^n}\binom{2n}{n} >0.
\]

A recurrent scheme for the calculation of the coefficients $c_n(\nu)$ may be derived as follows. Note that the asymptotic series \eqref{Xasymptoticseries} of $X_\nu(w)$ arises by means of formal inversion of the asymptotic series \eqref{Thetaasymptoticseries} of $\Theta_\nu(z)$. Therefore, by a theorem of Fabijonas and Olver \cite[Theorem 2.4]{Fabijonas1999}, $(1 - 2n)c_n (\nu )$ is the coefficient of $z^{-1}$ in the asymptotic expansion of $(\Theta_\nu(z))^{2n-1}$, i.e., the coefficient of $z^{-2n}$ in the asymptotic expansion of $(z^{-1}\Theta_\nu(z))^{2n-1}$. Let us denote, for each positive integer $n$, the coefficient of $z^{-2k}$ in the expansion of $(z^{-1}\Theta_\nu(z))^{2n-1}$ by $t_k^{(n)} (\nu )$, that is
\[
\left( 1 + \sum\limits_{k = 1}^\infty \frac{t_k (\nu )}{z^{2k}} \right)^{2n - 1}  \sim 1 + \sum\limits_{k = 1}^\infty \frac{t_k^{(n)} (\nu )}{z^{2k}}.
\]
With this notation,
\[
c_n (\nu ) = \frac{t_n^{(n)} (\nu )}{1 - 2n}.
\]
By an exercise of Olver \cite[Ch. 1, \S8, Ex. 8.4]{Olver1997}, the coefficients $t_k^{(n)} (\nu)$ may be computed by the recurrence relation
\[
t_1^{(n)} (\nu ) = (2n - 1)t_1 (\nu ), \qquad t_k^{(n)} (\nu ) = (2n - 1)t_k (\nu ) + \sum\limits_{j = 1}^{k - 1} \frac{2nj - k}{k}t_j (\nu )t_{k - j}^{(n)} (\nu ) 
\]
for all $n\geq 1$, $k \geq 2$ and real $\nu$. A simple induction argument reveals that $t_k^{(n)} (\nu )$ is a polynomial in $\nu^2$ of degree $k$ for each $1\leq k \leq n$. Hence, $c_n (\nu)$ is a polynomial in $\nu^2$ of degree $n$. We remark that a somewhat more involved algorithm for the evaluation of the coefficients $c_n (\nu)$ was given earlier by D\"oring \cite{Doring1966}. The explicit form of the first seven coefficients $c_n(\nu)$ can be found in \cite[p. xxxvii]{Bickley1952}.

A recursive formula for the sequence $T_n$ can be obtained from the definition \eqref{Tdef} of $T(w)$. If we replace $T(w)$ and $X_\nu(w)$ in \eqref{Tdef} by their respective asymptotic expansions, we arrive at the relation
\[
\left( 1 + \sum\limits_{n = 1}^\infty  \left( \frac{3}{2} \right)^{2n} c_n \left( \tfrac{1}{3} \right)\frac{1}{w^{2n}} \right)^{2/3}  \sim 1 + \sum\limits_{n = 1}^\infty \frac{T_n}{w^{2n}}.
\]
We can refer again to \cite[Ch. 1, \S8, Ex. 8.4]{Olver1997} and obtain
\[
T_1  = \frac{5}{48}, \qquad T_n  = \frac{2}{3}\left( \frac{3}{2} \right)^{2n} c_n \left( \tfrac{1}{3} \right) + \sum\limits_{k = 1}^{n - 1} \frac{5k - 3n}{3}\left( \frac{3}{2} \right)^{2k} c_k \left( \tfrac{1}{3}\right)T_{n - k} 
\]
for all $n\geq 2$. It is seen from this recurrence relation and the above algorithm for the coefficients $c_n(\nu)$, that the $T_n$'s are all rational numbers. For a list of the first nine coefficients $T_n$, see \cite{Fabijonas1999}.

\section{}\label{Appendix2}

In this appendix, we prove formula \eqref{Aviaphase}. To keep the derivation concise, we introduce the notation $\zeta  = \frac{2}{3}z^{3/2}$, $z>0$. First, we make use of the known relation between the Airy functions and the Bessel functions of order $\pm \frac{1}{3}$ (see, \cite[\href{http://dlmf.nist.gov/9.6.E6}{Eq. 9.6.6} and \href{http://dlmf.nist.gov/9.6.E8}{Eq. 9.6.8}]{DLMF}), and the representation \eqref{Jviaphase} to obtain
\begin{align*}
\mathscr{A}( - z,\alpha )  =\; & \frac{\sqrt z }{3}(J_{1/3} (\zeta ) + J_{ - 1/3} (\zeta ))\cos (\pi \alpha ) - \sqrt {\frac{z}{3}} (J_{1/3} (\zeta ) - J_{ - 1/3} (\zeta ))\sin (\pi \alpha )
\\  = \; &\frac{\sqrt z}{3}M_{1/3} (\zeta )(\cos \theta _{1/3} (\zeta ) + \cos \theta _{ - 1/3} (\zeta ))\cos (\pi \alpha ) \\ & - \sqrt {\frac{z}{3}} M_{1/3} (\zeta )(\cos \theta _{1/3} (\zeta ) - \cos \theta _{ - 1/3} (\zeta ))\sin (\pi \alpha ).
\end{align*}
Now, by \eqref{phaserelation} and basic trigonometry,
\[
\cos \theta _{ - 1/3} (\zeta ) = \cos \left( \theta _{1/3} (\zeta ) + \tfrac{\pi }{3} \right) = \tfrac{1}{2}\cos \theta _{1/3} (\zeta ) - \tfrac{\sqrt 3 }{2}\sin \theta _{1/3} (\zeta ).
\]
Hence, we arrive at
\begin{align*}
\mathscr{A}( - z,\alpha )  =\; & \sqrt {\frac{z}{3}} M_{1/3} (\zeta )\left( \tfrac{\sqrt 3}{2}\cos \theta _{1/3} (\zeta ) - \tfrac{1}{2}\sin \theta _{1/3} (\zeta ) \right)\cos (\pi \alpha ) \\ & - \sqrt {\frac{z}{3}} M_{1/3} (\zeta )\left( \tfrac{1}{2}\cos \theta _{1/3} (\zeta ) + \tfrac{\sqrt 3}{2}\sin \theta _{1/3} (\zeta ) \right)\sin (\pi \alpha ).
\end{align*}
Finally, we note that
\begin{align*}
& \tfrac{\sqrt 3}{2}\cos \theta _{1/3} (\zeta )\cos (\pi \alpha ) - \tfrac{1}{2}\sin \theta _{1/3} (\zeta )\cos (\pi \alpha ) - \tfrac{1}{2}\cos \theta _{1/3} (\zeta )\sin (\pi \alpha ) - \tfrac{\sqrt 3}{2}\sin \theta _{1/3} (\zeta )\sin (\pi \alpha )
\\ & = \tfrac{\sqrt 3}{2}(\cos \theta _{1/3} (\zeta )\cos (\pi \alpha ) - \sin \theta _{1/3} (\zeta )\sin (\pi \alpha )) - \tfrac{1}{2}(\sin \theta _{1/3} (\zeta )\cos (\pi \alpha ) + \cos \theta _{1/3} (\zeta )\sin (\pi \alpha ))
\\ & = \tfrac{\sqrt 3}{2}\cos (\theta _{1/3} (\zeta ) + \pi \alpha ) - \tfrac{1}{2}\sin (\theta _{1/3} (\zeta ) + \pi \alpha ) = \cos \left( \theta _{1/3} (\zeta ) + \pi \left( \alpha  + \tfrac{1}{6} \right) \right),
\end{align*}
giving the desired result \eqref{Aviaphase}.

\bigskip 

\end{document}